\documentclass{amsart}

\usepackage{amsthm}
\usepackage[leqno]{amsmath}
\usepackage[all]{xy} \SelectTips{eu}{}
\usepackage{latexsym,amsfonts,amssymb}
\usepackage{hyperref}
\usepackage{mathptmx}
\usepackage{nicefrac}

\newcommand{\numberseries}{\bfseries}   

\newlength{\thmtopspace}                
\newlength{\thmbotspace}                
\newlength{\thmheadspace}               
\newlength{\thmindent}                  

\newtheoremstyle{fixed bf head,slanted body}
                {\thmtopspace}{\thmbotspace}{\slshape}
                {\thmindent}{\bfseries}{.}{\thmheadspace}
                {{\numberseries \thmnumber{#2\;}}\thmname{#1}\thmnote{ (#3)}}

\newtheoremstyle{variable bf head,slanted body}
                {\thmtopspace}{\thmbotspace}{\slshape}
                {\thmindent}{\bfseries}{.}{\thmheadspace}
                {{\numberseries \thmnumber{#2\;}}\thmname{#1}\thmnote{ #3}}

\newtheoremstyle{fixed bf head,upright body}
                {\thmtopspace}{\thmbotspace}{\upshape}
                {\thmindent}{\bfseries}{.}{\thmheadspace}
                {{\numberseries \thmnumber{#2\;}}\thmname{#1}\thmnote{ (#3)}}

\newtheoremstyle{numbered paragraph}
                {\thmtopspace}{\thmbotspace}{\upshape}
                {\thmindent}{\upshape}{}{\thmheadspace}
                {{\numberseries \thmnumber{#2.}}}

\theoremstyle{fixed bf head,slanted body}
\newtheorem{res}{}[section]
\newtheorem{thm}[res]{Theorem}          \newtheorem*{thm*}{Theorem}
\newtheorem{prp}[res]{Proposition}      \newtheorem*{prp*}{Proposition}
        \newtheorem*{cor*}{Corollary}
\newtheorem{lem}[res]{Lemma}            \newtheorem*{lem*}{Lemma}

\theoremstyle{variable bf head,slanted body}
     \newtheorem*{introthm*}{Theorem}

\theoremstyle{fixed bf head,upright body}
\newtheorem{stp}[res]{Setup}            \newtheorem*{stp*}{Setup}
\newtheorem{dfn}[res]{Definition}       \newtheorem*{dfn*}{Definition}
     \newtheorem*{con*}{Construction}
\newtheorem{obs}[res]{Observation}      \newtheorem*{obs*}{Observation}
\newtheorem{rmk}[res]{Remark}           \newtheorem*{rmk*}{Remark}
\newtheorem{exa}[res]{Example}          \newtheorem*{exa*}{Example}
\newtheorem{qst}[res]{Question}         \newtheorem*{qst*}{Question}

\theoremstyle{numbered paragraph}

\newlength{\thmlistleft}        
\newlength{\thmlistright}       
\newlength{\thmlistpartopsep}   
\newlength{\thmlisttopsep}      
\newlength{\thmlistparsep}      
\newlength{\thmlistitemsep}     

\newcounter{eqc} 
\newenvironment{eqc}{\begin{list}{\upshape (\textit{\roman{eqc}})}%
    {\usecounter{eqc}%
      \setlength{\leftmargin}{\thmlistleft}%
      \setlength{\labelwidth}{\thmlistleft}%
      \setlength{\rightmargin}{\thmlistright}%
      \setlength{\partopsep}{\thmlistpartopsep}%
      \setlength{\topsep}{\thmlisttopsep}%
      \setlength{\parsep}{\thmlistparsep}%
      \setlength{\itemsep}{\thmlistitemsep}}}%
  {\end{list}}%

\newcommand{\eqclbl}[1]{{\upshape(\textit{#1})}}

\newcounter{prt}
\newenvironment{prt}{\begin{list}{\upshape (\alph{prt})}%
    {\usecounter{prt}%
      \setlength{\leftmargin}{\thmlistleft}%
      \setlength{\labelwidth}{\thmlistleft}%
      \setlength{\rightmargin}{\thmlistright}%
      \setlength{\partopsep}{\thmlistpartopsep}%
      \setlength{\topsep}{\thmlisttopsep}%
      \setlength{\parsep}{\thmlistparsep}%
      \setlength{\itemsep}{\thmlistitemsep}}}%
  {\end{list}}%

  \newcommand{\proofofimp}[3][:]{\mbox{\eqclbl{#2}$\!\implies\!$\eqclbl{#3}#1}}

\newcommand{\pgref}[1]{\ref{#1}}
\newcommand{\thmref}[2][Theorem~]{#1\pgref{thm:#2}}

\newcommand{\prpref}[2][Proposition~]{#1\pgref{prp:#2}}
\newcommand{\lemref}[2][Lemma~]{#1\pgref{lem:#2}}
\newcommand{\obsref}[2][Observation~]{#1\pgref{obs:#2}}

\newcommand{\dfnref}[2][Definition~]{#1\pgref{dfn:#2}}
\newcommand{\exaref}[2][Example~]{#1\pgref{exa:#2}}
\newcommand{\rmkref}[2][Remark~]{#1\pgref{rmk:#2}}
\newcommand{\secref}[2][Section~]{#1\ref{sec:#2}}

\renewcommand{\eqref}[1]{(\pgref{eq:#1})}

\newcommand{\stpref}[2][Setup~]{#1\pgref{stp:#2}}

\makeatletter
\def\@nobreak@#1{\mathchoice%
  {\nobreakdef@\displaystyle\f@size{#1}}%
  {\nobreakdef@\nobreakstyle\tf@size{\firstchoice@false #1}}%
  {\nobreakdef@\nobreakstyle\sf@size{\firstchoice@false #1}}%
  {\nobreakdef@\nobreakstyle\ssf@size{\firstchoice@false #1}}%
  \check@mathfonts}%
\def\nobreakdef@#1#2#3{\hbox{{%
                    \everymath{#1}%
                    \let\f@size#2\selectfont%
                    #3}}}%
\makeatother

\DeclareSymbolFont{usualmathcal}{OMS}{cmsy}{m}{n}
\DeclareSymbolFontAlphabet{\mathcal}{usualmathcal}

\DeclareSymbolFont{letters}{OML}{txmi}{m}{it}

\DeclareMathSymbol{\alpha}{\mathord}{letters}{"0B}
\DeclareMathSymbol{\beta}{\mathord}{letters}{"0C}
\DeclareMathSymbol{\gamma}{\mathord}{letters}{"0D}
\DeclareMathSymbol{\delta}{\mathord}{letters}{"0E}
\DeclareMathSymbol{\epsilon}{\mathord}{letters}{"0F}
\DeclareMathSymbol{\zeta}{\mathord}{letters}{"10}
\DeclareMathSymbol{\eta}{\mathord}{letters}{"11}
\DeclareMathSymbol{\theta}{\mathord}{letters}{"12}
\DeclareMathSymbol{\iota}{\mathord}{letters}{"13}
\DeclareMathSymbol{\kappa}{\mathord}{letters}{"14}
\DeclareMathSymbol{\lambda}{\mathord}{letters}{"15}
\DeclareMathSymbol{\mu}{\mathord}{letters}{"16}
\DeclareMathSymbol{\nu}{\mathord}{letters}{"17}
\DeclareMathSymbol{\xi}{\mathord}{letters}{"18}
\DeclareMathSymbol{\pi}{\mathord}{letters}{"19}
\DeclareMathSymbol{\rho}{\mathord}{letters}{"1A}
\DeclareMathSymbol{\sigma}{\mathord}{letters}{"1B}
\DeclareMathSymbol{\tau}{\mathord}{letters}{"1C}
\DeclareMathSymbol{\upsilon}{\mathord}{letters}{"1D}
\DeclareMathSymbol{\phi}{\mathord}{letters}{"1E}
\DeclareMathSymbol{\chi}{\mathord}{letters}{"1F}
\DeclareMathSymbol{\psi}{\mathord}{letters}{"20}
\DeclareMathSymbol{\omega}{\mathord}{letters}{"21}
\DeclareMathSymbol{\varepsilon}{\mathord}{letters}{"22}
\DeclareMathSymbol{\vartheta}{\mathord}{letters}{"23}
\DeclareMathSymbol{\varpi}{\mathord}{letters}{"24}
\DeclareMathSymbol{\varrho}{\mathord}{letters}{"25}
\DeclareMathSymbol{\varsigma}{\mathord}{letters}{"26}
\DeclareMathSymbol{\varphi}{\mathord}{letters}{"27}

\DeclareMathSymbol{\Gamma}{\mathord}{letters}{"00}
\DeclareMathSymbol{\Delta}{\mathord}{letters}{"01}
\DeclareMathSymbol{\Theta}{\mathord}{letters}{"02}
\DeclareMathSymbol{\Lambda}{\mathord}{letters}{"03}
\DeclareMathSymbol{\Xi}{\mathord}{letters}{"04}
\DeclareMathSymbol{\Pi}{\mathord}{letters}{"05}
\DeclareMathSymbol{\Sigma}{\mathord}{letters}{"06}
\DeclareMathSymbol{\Upsilon}{\mathord}{letters}{"07}
\DeclareMathSymbol{\Phi}{\mathord}{letters}{"08}
\DeclareMathSymbol{\Psi}{\mathord}{letters}{"09}
\DeclareMathSymbol{\Omega}{\mathord}{letters}{"0A}

\DeclareMathSymbol{\upGamma}{\mathalpha}{operators}{"00}
\DeclareMathSymbol{\upDelta}{\mathalpha}{operators}{"01}
\DeclareMathSymbol{\upTheta}{\mathalpha}{operators}{"02}
\DeclareMathSymbol{\upLambda}{\mathalpha}{operators}{"03}
\DeclareMathSymbol{\upXi}{\mathalpha}{operators}{"04}
\DeclareMathSymbol{\upPi}{\mathalpha}{operators}{"05}
\DeclareMathSymbol{\upSigma}{\mathalpha}{operators}{"06}
\DeclareMathSymbol{\upUpsilon}{\mathalpha}{operators}{"07}
\DeclareMathSymbol{\upPhi}{\mathalpha}{operators}{"08}
\DeclareMathSymbol{\upPsi}{\mathalpha}{operators}{"09}
\DeclareMathSymbol{\upOmega}{\mathalpha}{operators}{"0A}

\newcommand{\MCM}{\mathsf{MCM}}
\newcommand{\I}{\mathsf{I}}
\newcommand{\Ker}[1]{\operatorname{Ker}#1}
\renewcommand{\Im}[1]{\operatorname{Im}#1}
\newcommand{\Hom}[3][R]{\operatorname{Hom}_{#1}(#2,#3)}
\newcommand{\Ext}[4][R]{\operatorname{Ext}_{#1}^{#2}(#3,#4)}
\newcommand{\Tor}[4][R]{\operatorname{Tor}^{\mspace{2mu}#1}_{#2}(#3,#4)}
\newcommand{\depthR}[1][R]{\operatorname{depth}\mspace{2mu}#1}
\newcommand{\depth}[2][R]{\operatorname{depth}_{#1}#2}
\renewcommand{\dim}[1][R]{\operatorname{dim}\,#1}
\newcommand{\dual}{\upOmega}
\newcommand{\te}{R \ltimes \dual}
\newcommand{\m}{\mathfrak{m}}
\newcommand{\n}{\mathfrak{n}}
\newcommand{\p}{\mathfrak{p}}
\newcommand{\q}{\mathfrak{q}}
\newcommand{\Koszul}[2]{\operatorname{K}_\bullet^{#1}(#2)}
\newcommand{\Ltp}[3][R]{#2\otimes_{#1}^{\mathbf{L}}#3}
\newcommand{\wbbCM}{\mathsf{wbbCM}}

\begin{document}

\title[The structure of balanced big CM modules over CM rings]{The structure of balanced big Cohen--Macaulay modules over Cohen--Macaulay rings}

\author{Henrik Holm}

\address{University of Copenhagen, 2100 Copenhagen {\O}, Denmark}
 
\email{holm@math.ku.dk}

\urladdr{http://www.math.ku.dk/\~{}holm/}





\begin{abstract}
  Over a Cohen--Macaulay (CM) local ring, we characterize those modules that can be obtained as a direct limit of finitely generated maximal CM modules. We point out two consequences of this characterization: (1) Every balanced big CM module, in the sense of Hochster, can be written as a direct limit of small CM modules. In analogy with Govorov and Lazard's characterization of flat modules as direct limits of finitely generated free modules, one can view this as a ``structure theorem'' for balanced big CM modules. (2) Every finitely generated module has a preenvelope with respect to the class of finitely generated maximal CM modules. This result is, in some sense, dual to the existence of maximal CM approximations, which is proved by Auslander and Buchweitz.
\end{abstract}

\maketitle


\section{Introduction}

Let $R$ be a local ring. Hochster \cite{Hochster} defines an $R$-module $M$~to be \emph{big Cohen--Macaulay} (big CM) if some system of parameters (s.o.p.)~of $R$ is an $M$-regular sequence. If every s.o.p.~of $R$ is an $M$-regular sequence, then $M$ is called \emph{balanced big CM}. The term ``big'' refers to the fact that $M$ need not be finitely generated; and a finitely generated (balanced) big CM module is called a \emph{small CM module}. It is conjectured by Hochster, see (2.1) in \emph{loc.~cit.}, that every local ring has a big CM module. This conjecture is still open, however, it has been settled affirmatively by Hochster \cite{Hochster75b, Hochster75a} in the equicharacteristic case, that is, for local rings containing a field. In fact, such a ring always has a balanced big CM~module. 

Although this paper makes no contribution to Hochster's conjecture, it is concerned with balanced big CM modules. We study such modules over a CM ring $R$ with a \mbox{dualizing} module $\dual$.
In this setting, the conjecture is of course trivially true since small, and hence also balanced big, CM modules abound (for example, $R$ and $\dual$ are such modules). It turns out that all balanced big CM $R$-modules share a common ``structure'': they can always be build from small ones. The following result is proved in \secref{direct-limit-closure}.

\begin{introthm*}[A]
Every balanced big CM $\,R$-module is a direct limit of small CM $\,R$-modules.
\end{introthm*}

\exaref{Griffith1} exhibits a (non-balanced) big CM module that is \emph{not} a direct limit of small CM modules. 

A \emph{finitely generated maximal CM module} is a module which is either small CM or zero; a covention used by Auslander and Buchweitz \cite{MAsROB89} and others. Theorem~A is a consequence of the next result---also proved in \secref{direct-limit-closure}---which gives two equivalent characterizations of the direct limit closure of the class of finitely generated maximal CM modules.

\begin{introthm*}[B]
  For every $R$-module $M$, the following conditions are equivalent.
\begin{eqc}
\item $M$ is a direct limit of finitely generated maximal CM $\,R$-modules.
\item Every system of parameters of $R$ is a weak $M$-regular sequence.
\item $M$ is Gorenstein flat\footnote{\ In the sense of Enochs, Jenda, and Torrecillas \cite{EJT-93}; see \dfnref{G-flat}. } viewed as a module over the trivial extension $\te$.
\end{eqc}
\end{introthm*}

This theorem is analogous to a classic result, due to Govorov \cite{VEG65} and Lazard \cite{DLz69},~which shows that the direct limit closure of the class of finitely generated free modules is precisely the class of flat modules. Following Hochster's terminology, it is reasonable to call a module \emph{weak balanced big CM} if it satisfies condition \eqclbl{ii} in Theorem~B. This class of modules is denoted by $\wbbCM$, and it is a natural extension of the class of finitely generated maximal CM modules to the realm of all modules. Indeed, by Nakayama's lemma, a finitely generated module belongs to $\wbbCM$ if and only if it is maximal CM.

In \secref{rha} we give applications of Theorem~B in relative homological algebra. It follows from works of Ischebeck \cite{FIs69} and Auslander and Buchweitz \cite{MAsROB89} that the class $\MCM$ of finitely generated maximal CM $R$-modules is part of a complete hereditary cotorsion pair $(\MCM,\MCM^\perp)$ on the category of finitely generated $R$-modules\footnote{\ Actually, $\MCM^\perp$ is the class of finitely generated $R$-modules with finite injective dimension; cf.~\thmref[Thm.~]{MCM-cotorsion-pair}.}. In particular, every finitely generated $R$-module has an $\MCM$-precover and an $\MCM^\perp$-preenvelope. We show:

\begin{introthm*}[C]
  Every finitely generated $R$-module has an $\MCM$-preenvelope.
\end{introthm*}

We also extend the cotorsion pair $(\MCM,\MCM^\perp)$ and the existence of $\MCM$-preenve\-lopes to the realm of all---not necessarily finitely generated---modules:

\begin{introthm*}[D]
  On the category of all $R$-modules, $(\wbbCM,\wbbCM^\perp)$ is a perfect hereditary cotorsion pair, in particular, every $R$-module has a $\wbbCM$-cover and a $\wbbCM^\perp$-envelope. Furthermore, every $R$-module has a $\wbbCM$-preenvelope. 
\end{introthm*}

As a consequence of Theorem~D, we prove the existence of (non-weak) balanced big~CM covers for certain types of modules; see \prpref{bbCM-cover} and \exaref{bbCM-cover}.

\section{Regular sequences, depth, and CM modules}
\label{sec:depth}

\begin{stp}
  \label{stp:setup}
  Throughout, $(R,\mathfrak{m},k)$ is a commutative noetherian local CM ring with Krull dimension $d$. It is assumed that $R$ has a dualizing (or canonical) module $\dual$. 
\end{stp}

Let $(A,\n,\ell)$ be any commutative noetherian local ring and let $M$ be any $A$-module. A sequence of elements $\bar{x}=x_1,\ldots,x_n \in \n$ is called a \emph{weak $M$-regular sequence} if $x_i$ is a non-zerodivisor on $M/(x_1,\ldots,x_{i-1})M$ for every $1 \leqslant i \leqslant n$ (for $i=1$ this means that $x_1$ is a non-zerodivisor on $M$). If, in addition, $(x_1,\ldots,x_n)M \neq M$, then $\bar{x}$ is an \emph{$M$-regular sequence}. If $M \neq 0$ is finitely generated, then by Nakayama's lemma every weak $M$-regular sequence is automatically $M$-regular; this is not the case in general. 

The \emph{depth} of a finitely generated $A$-module $M \neq 0$, denoted by $\depth[A]{M}$, is the supremum of the lengths of all $M$-regular sequences (alternatively, the
common length of all maximal $M$-regular sequences). This invariant can be computed homologically as follows:
\begin{equation}
  \label{eq:depth}
  \depth[A]{M} \,=\, \inf\{\,i \in \mathbb{Z} \,|\, \Ext[A]{i}{\ell}{M} \neq 0\,\}\;.
\end{equation}
For an arbitrary $A$-module $M$, we \emph{define} its depth\footnote{\ Some authors refer to the number in \eqref{depth} as the ``Ext-depth'' of $M$. If $\n M \neq M$, then
$\depth[A]{M}$ (i.e.~the ``Ext-depth'' of $M$) is always an upper bound for the length of any $M$-regular sequence; see Strooker \cite[Prop.~5.3.7(ii)]{Strooker}. However, if $M$ is not finitely generated, then there does not necessarily exist an $M$-regular sequence of length $\depth[A]{M}$; see p.~91 in \emph{loc.~cit.}~for a counterexample.} by the formula \eqref{depth}, with the convention that $\inf\,\emptyset = \infty$. So, for example, the zero module has infinite depth\footnote{\ It is also possible for a non-zero module to have infinite depth; see \obsref{Griffith2}.}.

For a finitely generated $A$-module $M \neq 0$ one always has $\depth[A]{M} \leqslant \dim[A]$, and the following conditions are equivalent; see
Eisenbud \cite[Prop.-Def.~21.9]{Eis}.
\begin{eqc}
\item $\depth{M}=\dim[A]$.
\item Every system of parameters of $A$ is an $M$-regular sequence.
\item Some system of parameters of $A$ is an $M$-regular sequence.
\end{eqc}
A finitely generated module $M$ that satisfies these equivalent conditions is called \emph{small CM}. A \emph{finitely generated maximal CM module} is a module which is either small CM or zero, and the category of all such modules is denoted by $\MCM$. Unlike the category of small CM modules, the category $\MCM$ additive and closed under direct summands. Some authors, such as Yoshino \cite{Yos}, use the simpler terminology ``CM module'' for what we have called a ``maximal CM module''.

It is well-known that for an arbitrary $A$-module $M$, the conditions \eqclbl{i}--\eqclbl{iii} above are no longer equivalent, and hence there is more than one way to extend the notion of ``(maximal) CM'' to the realm of non-finitely generated modules. The next definition is due to Hochster \cite{Hochster} (the term ``balanced'' appears in Sharp \cite{Sharp81}).

\begin{dfn}
  \label{dfn:bigCM}
An $A$-module $M$ is called \emph{(balanced) big CM} if (every) some system of parameters of $A$ is an $M$-regular sequence.
\end{dfn}

It is well-known that a big CM module need not be balanced; cf.~\exaref{Griffith1}. As noted above, a finitely generated module is big CM if and only if it is balanced big CM.

\begin{lem}
  \label{lem:Tor}
  Let $M$ be any (not necessarily finitely generated) $R$-module and assume that $\bar{x}=x_1,\ldots,x_n$ is both an $R$-regular and a weak $M$-regular sequence. Then one has
\begin{displaymath}
  \Tor{i}{R/(\bar{x})}{M}=0 \quad \text{ for all } \quad i>0\;.
\end{displaymath}
\end{lem}

\begin{proof}
  By induction on $n$. For $n=1$ we have a single element $x_1$ which is a non-zerodivisor on both $R$ and $M$. The assertion follows from inspection of the long exact Tor-sequence that arises from application of $-\otimes_RM$ to the short exact sequence \smash{$0 \to R \stackrel{x_1}{\to} R \to R/(x_1) \to 0$}.

Next we assume that $n>0$. Consider the ring $\bar{R} = R/(x_1,\ldots,x_{n-1})$ and the $\bar{R}$-module $\bar{M} = \bar{R} \otimes_R M = M/(x_1,\ldots,x_{n-1})M$. By the induction hypothesis, 
\begin{displaymath}
  \Tor{i}{\bar{R}}{M} \,=\, \Tor{i}{R/(x_1,\ldots,x_{n-1})}{M} \,=\, 0 \quad \text{ for all } \quad i>0\;.
\end{displaymath}
Thus, in the derived category over $R$, one has $\Ltp{\bar{R}}{M} \cong \bar{R} \otimes_R M = \bar{M}$, and consequently
\begin{displaymath}
  \Ltp{R/(x_1,\ldots,x_n)}{M}
  \,=\, 
  \Ltp{\bar{R}/(\bar{x}_n)}{M}
  \,\cong\, 
  \Ltp[\bar{R}]{\bar{R}/(\bar{x}_n)}{(\Ltp{\bar{R}}{M})} 
  \,\cong\, 
  \Ltp[\bar{R}]{\bar{R}/(\bar{x}_n)}{\bar{M}}\;.
\end{displaymath}
In particular, $\Tor{i}{R/(x_1,\ldots,x_n)}{M} \cong \Tor[\bar{R}]{i}{\bar{R}/(\bar{x}_n)}{\bar{M}}$ for every $i>0$. The latter Tor is $0$; this follows from the induction start 
as $\bar{x}_n \in \bar{R}$ is a non-zerodivisor on both $\bar{R}$ and $\bar{M}$.
\end{proof}

\begin{prp}
  \label{prp:sop}
  For every (not necessarily finitely generated) $R$-module $M$, the following conditions are equivalent:
\begin{eqc}
\item Every system of parameters of $R$ is a weak $M$-regular sequence.
\item For every $R$-regular sequence $\bar{x}$ one has $\Tor{i}{R/(\bar{x})}{M}=0$ for all $i>0$.
\item For every $R$-regular sequence $\bar{x}$ one has $\Tor{1}{R/(\bar{x})}{M}=0$.
\end{eqc}
\end{prp}

\begin{proof}
\proofofimp{i}{ii} Let $\bar{x}$ be any $R$-regular sequence. As $R$ is CM, $\bar{x}$ is part of a s.o.p.~of $R$; see \cite[Thm.~2.1.2(d)]{BruHer}. By the assumption \eqclbl{i}, this s.o.p.~is a weak $M$-regular sequence, and hence so is the subsequence $\bar{x}$. \lemref{Tor} gives the desired conclusion.

\proofofimp{ii}{iii} Clear.

\proofofimp{iii}{i} Let $\bar{x}=x_1,\ldots,x_d$ be any s.o.p.~of $R$. Since $R$ is CM, the sequence $\bar{x}$ is $R$-regular; see \secref{depth}. Thus, for every $i=1,\ldots,d$ there is an exact sequence,
\begin{displaymath}
  0 \longrightarrow R/(x_1,\ldots,x_{i-1}) \stackrel{x_i}{\longrightarrow} R/(x_1,\ldots,x_{i-1}) \longrightarrow R/(x_1,\ldots,x_i) \longrightarrow 0\;.
\end{displaymath}
Application of the functor $-\otimes_RM$ to this sequence yields the long exact sequence
\begin{displaymath}
  \Tor{1}{R/(x_1,\ldots,x_i)}{M} \longrightarrow M/(x_1,\ldots,x_{i-1})M \stackrel{x_i}{\longrightarrow} M/(x_1,\ldots,x_{i-1})M 
\end{displaymath}
The sequence $x_1,\ldots,x_i$ is $R$-regular, as it is a subsequence of the $R$-regular sequence $\bar{x}$. Hence $\Tor{1}{R/(x_1,\ldots,x_i)}{M}=0$ by the assumption \eqclbl{iii}, and thus $x_i$ is a non-zerodivisor on $M/(x_1,\ldots,x_{i-1})M$. Therefore $\bar{x}$ is a weak $M$-regular sequence.

\end{proof}

\section{The trivial extension}

Let $A$ be any commutative ring and let $C$ be any $A$-module. The \emph{trivial extension} of $A$ by $C$ (also called the \emph{idealization} of $C$ in $A$) is the ring $A \ltimes C$ whose underlying abelian group is $A \oplus C$ and where multiplication is given by
\begin{displaymath}
 (a,c)(a',c') \,=\, (aa',ac'+a'c) \quad \text{ for } \quad (a,c),(a',c') \in A \oplus C\;.
\end{displaymath}
We refer to Fossum, Griffith, and Reiten \cite[\S 5]{FGR-75} for basic properties of this construction. The ring homomorphisms $\varphi \colon A \to A \ltimes C$ given by $a \mapsto (a,0)$, and $\psi \colon A \ltimes C \to A$ given by $(a,c) \mapsto a$ allow us to consider any $A$-module as an $(A \ltimes C)$-module, and vice versa, and we shall do so freely. Note that the composition $\psi\varphi$ is the identity on $A$.

\begin{lem}
  \label{lem:te-1}
  Let $A$ be any commutative ring. For any $A$-module $C$ and any set of elements $x_1,\ldots,x_n$ in $A$ there is an isomorphism of rings,
  \begin{displaymath}
    (A \ltimes C)/((x_1,0),\ldots,(x_n,0)) \ \cong \ A/(x_1,\ldots,x_n) \ltimes C/(x_1,\ldots,x_n)C\;.
  \end{displaymath}
\end{lem}

\begin{proof}
  There is a surjective homomorphism $\varphi \colon A \ltimes C \twoheadrightarrow A/(x_1,\ldots,x_n) \ltimes C/(x_1,\ldots,x_n)C$ given by $(a,c) \mapsto ([a]_{(x_1,\ldots,x_n)},[c]_{(x_1,\ldots,x_n)C})$. Clearly, we have $(x_i,0) \in \Ker{\varphi}$. Conversely, if $(a,c) \in \Ker{\varphi}$, then $a=\sum_{i=1}^nx_ia_i$ and $c=\sum_{i=1}^nx_ic_i$ where $a_i \in A$ and $c_i \in C$. It fol\-lows that $(a,c) = \sum_{i=1}^n(x_i,0)(a_i,c_i)$, so $(a,c)$ is in the ideal $((x_1,0),\ldots,(x_n,0))$ in $A \ltimes C$.
\end{proof}

While the previous lemma was quite general, the next one is more specific.

\begin{lem}
  \label{lem:te-2}
  Let $C \neq 0$ be any finitey generated maximal CM $R$-module. If $x_1,\ldots,x_n \in \m$ is an $R$-regular sequence, then $(x_1,0),\ldots,(x_n,0)$ is an $(R \ltimes C)$-regular sequence.
\end{lem}

\begin{proof}
Fix $i \in \{1,\ldots,n\}$. As $C \neq 0$ is a maximal CM $R$-module and the sequence $x_1,\ldots,x_{i-1}$ is $R$-regular, $\bar{C}=C/(x_1,\ldots,x_{i-1})C \neq 0$ is a maximal CM module over $\bar{R}=R/(x_1,\ldots,x_{i-1})$. By assumption, $x_i$ is a non-zerodivisor on $\bar{R}$. We must argue that $(x_i,0)$ is a non-zerodivisor on $(R \ltimes C)/((x_1,0),\ldots,(x_{i-1},0)) \cong \bar{R} \ltimes \bar{C}$, where the isomorphism is by \lemref{te-1}. Thus, let $(r,c) \in \bar{R} \ltimes \bar{C}$ be any element such that $(x_i,0)(r,c) = (x_ir,x_ic)$ is $(0,0)$; that is, we have $x_ir=0$ and $x_ic=0$. By asumption, $x_i$ is a non-zerodivisor on $\bar{R}$, and since the $\bar{R}$-module $\bar{C}$ is maximal CM, the element $x_i$ is also a non-zerodivisor on $\bar{C}$. Thus the equations $x_ir=0$ and $x_ic=0$ imply that $r=0$ and $c=0$, that is, $(r,c)=(0,0)$ as desired.
\end{proof}

Let $A$ be a commutative ring and let $x \in A$ be an element. Recall that the \emph{Koszul complex} on $x$ is the complex \smash{$\Koszul{A}{x} \,=\, 0 \to A \stackrel{x}{\to} A \to 0$} concentrated in homological degrees $0,1$. For a sequence $\bar{x} = x_1,\ldots,x_n \in A$ the Koszul complex is $\Koszul{A}{\bar{x}} \,=\, \Koszul{A}{x_1} \otimes_A \cdots \otimes_A \Koszul{A}{x_n}$.

\begin{lem}
  \label{lem:te-3}
  Let $A$ be any commutative ring, let $C$ be an $A$-module, and let $\bar{x}=x_1,\ldots,x_n$ be a sequence of elements in~$A$. Consider the elements $y_i = (x_i,0)$ and the sequence $\bar{y}=y_1,\ldots,y_n$ in $A \ltimes C$. For every $A$-module $M$ there is the following isomorphism of both $A$- and $(A \ltimes C)$-complexes,
\begin{displaymath} 
  \Koszul{A \ltimes C}{\bar{y}} \otimes_{A \ltimes C} M \,\cong\, \Koszul{A}{\bar{x}} \otimes_A M\,.
\end{displaymath}
\end{lem}

\begin{proof}
  It suffices to consider the case $n=1$. We have
\begin{align*}
  \Koszul{A \ltimes C}{y_1} \otimes_{A \ltimes C} M &\ \cong \
  0 \longrightarrow M \stackrel{y_1}{\longrightarrow} M \longrightarrow 0
  \qquad \text{($M$ is viewed as an $(A \ltimes C)$-module)}
  \\
  \Koszul{A}{x_1} \otimes_A M &\ \cong \
  0 \longrightarrow M \stackrel{x_1}{\longrightarrow} M \longrightarrow 0
  \qquad \text{($M$ is viewed as an $A$-module)}
\end{align*}
By definition of the $(A \ltimes C)$-action on $M$, multiplication by the element $y_1=(x_1,0)$ on $M$ is nothing but multiplication by $x_1$ on $M$.
\end{proof}

\section{The direct limit closure of maximal CM modules}
\label{sec:direct-limit-closure}

By a \emph{filtered colimit} of maximal CM $R$-modules we mean the colimit of a functor $F$ from a skeletally small filtered category $\mathcal{J}$ to the category of $R$-modules such that $F(J)$ is maximal CM for every $J$ in $\mathcal{J}$. We reserve the term \emph{direct limit} for the special situation where $\mathcal{J}$ is the filtered category associated to a
\emph{directed set}, i.e.~a filtered preordered set. 

\begin{rmk}
\label{rmk:filtered-colimit}
It follows from general principles that a module is a filtered colimit of maximal CM modules if and only if it is a direct limit of maximal CM modules; see Ad{\'a}mek and Rosick{\'y} \cite[Thm.~1.5]{AdamekRosicky}. Thus in Theorem~B condition \eqclbl{i}, one can freely replace ``direct limit'' with ``filtered colimit''.
\end{rmk}

In Theorem~B condition \eqclbl{iii}, we encounter the notion of Gorenstein flat modules. These modules were defined by Enochs, Jenda, and Torrecillas \cite{EJT-93} as follows:

\begin{dfn}
 \label{dfn:G-flat}
 Let $A$ be any commutative ring. An $A$-module $M$ is called \emph{Gorenstein flat} if there exists an exact sequence of flat $A$-modules, $\mathbb{S} = \cdots \to F_1 \to F_0 \to F_{-1} \to \cdots$, with the property that $E \otimes_A \mathbb{S}$ is exact for every injective $A$-module $E$, such that $M \cong \Im{(F_0 \to F_{-1})}$.
\end{dfn}

We are now in a position to prove our main result: Theorem~B from the Introduction.

\begin{proof}[Proof of Theorem~B]
\proofofimp{i}{ii} Every finitely generated maximal CM $R$-module satisfies condition \eqclbl{ii}; see \secref{depth}. And \prpref{sop} shows that the class of $R$-modules that satisfy condition \eqclbl{ii} is closed under direct limits.

\proofofimp{ii}{iii} As $\dual$ is a dualizing $R$-module, $\te$ is a Gorenstein ring by \cite[Thm.~5.6]{FGR-75}. So at least the $(\te)$-module $M$ has finite Gorenstein flat dimension; see \cite[Cor.~2.4]{EJX-96b} or \cite[Thm.~5.2.10]{lnm}. It therefore follows from \cite[Cor.~5.4.9]{lnm} that this dimension, $\operatorname{Gfd}_{\te} M$, can be computed by the ``Chouinard formula'':
\begin{displaymath}
  \operatorname{Gfd}_{\te} M \,=\, \sup\{\,\depthR[(\te)_\q] - \depth[(\te)_\q]{M_\q} \ | \ \q \in \operatorname{Spec}(\te) \,\}\;.
\end{displaymath}
Thus, to prove that $M$ is Gorenstein flat over $\te$, equivalently, that $\operatorname{Gfd}_{\te} M \leqslant 0$ (the Gorenstein flat dimension of the zero module is $-\infty$), we must argue that the inequality
\begin{equation}
  \label{eq:depth-leq-zero}
  \depthR[(\te)_\q] - \depth[(\te)_\q]{M_\q} \,\leqslant\, 0
\end{equation}
holds for every prime ideal $\q$ in $\te$. By \cite[Lem.~5.1(i)]{FGR-75} every such $\q$ has the form $\q = \p \ltimes \dual$ for a prime ideal $\p$ in $R$. The rings $R$ and $\te$ are CM, and hence so are their localizations $R_\p$ and $(\te)_\q$. By Lem.~5.1(ii) in \emph{loc.~cit.}~the rings $R_\p$ and $(\te)_\q$ have the same Krull dimension, and this number we denote by $e$.

Recall that all maximal $R$-regular sequences contained in $\p$ have the same length; this number is called the \emph{grade of $\p$ on $R$} and it is denoted by $\operatorname{grade}_R{(\p,R)}$. Since $R$ is CM, we have $\operatorname{grade}_R{(\p,R)} = \depthR[R_\p] = e$ by \cite[Thm.~2.1.3(b)]{BruHer}. Now, let $\bar{x} = x_1,\ldots,x_e \in \p$ be a maximal $R$-regular sequence in $\p$. Set $y_i = (x_i,0)$ and $\bar{y} = y_1,\ldots,y_t$. The sequence $\bar{y}$ is evidently contained in $\q = \p \ltimes \dual$, and it is $(\te)$-regular by \lemref{te-2}. Hence $\bar{y}$ (or more precisely, the sequence $\nicefrac{y_1}{1},\ldots,\nicefrac{y_e}{1}$) is also $(\te)_\q$-regular, see \cite[Cor.~1.1.3(a)]{BruHer}. As noted above, the ring $(\te)_\q$ has depth (and Krull dimension) equal to $e$, and thus the $(\te)_\q$-module $(\te)_\q/(\bar{y})_\q \cong ((\te)/(\bar{y}))_\q$ has depth $0$, which means that the maximal ideal $\q_\q$ in $(\te)_\q$ is an associated prime of this module. It follows that $\q$ is an associated prime ideal of the $(\te)$-module $(\te)/(\bar{y})$; see \cite[Thm.~6.2]{Mat}. Note that the $(\te)$-module $(\te)/(\bar{y})$ has finite projective dimension (equal to $e$); this follows from \cite[Exerc.~1.3.6]{BruHer} and the fact that $\bar{y}$ is $(\te)$-regular. We are therefore in a position to apply \cite[Lem.~5.3.5(b)]{lnm}, which gives an inequality,
\begin{equation}
  \label{eq:depth-leq-sup-Tor}
  \depthR[(\te)_\q] - \depth[(\te)_\q]{M_\q} \,\leqslant\, \sup\{\, i \ | \, \Tor[\te]{i}{(\te)/(\bar{y})}{M} \neq 0 \,\}\;.
\end{equation}
Since $\bar{y}$ is an $(\te)$-regular sequence, the Koszul complex $\Koszul{\te}{\bar{y}}$ is a projective resolution of the $(\te)$-module $(\te)/(\bar{y})$; see
\cite[Thm.~16.5(i)]{Mat}. Similarly, $\Koszul{R}{\bar{x}}$ is a projective resolution of the $R$-module $R/(\bar{x})$. This explains the first and last isomorphism below; the middle isomorphism follows from \lemref{te-3}:
\begin{align*}
  \Tor[\te]{i}{(\te)/(\bar{y})}{M} 
  &\,\cong\, 
  \operatorname{H}_i(\Koszul{\te}{\bar{y}} \otimes_{\te} M)
  \\
  &\,\cong\, 
  \operatorname{H}_i(\Koszul{R}{\bar{x}} \otimes_R M)
  \\
  &\,\cong\, 
  \Tor{i}{R/(\bar{x})}{M}\;.
\end{align*}
The assumption \eqclbl{ii} and \prpref{sop} shows that $\Tor{i}{R/(\bar{x})}{M}=0$ for all $i>0$. This fact, combined with \eqref{depth-leq-sup-Tor}, gives the desired conclusion \eqref{depth-leq-zero}. 

\proofofimp{iii}{i} Recall from \secref{depth} that the category of finitely generated maximal CM $R$-modules is an additive category closed under direct summands. To prove \eqclbl{i} we apply Lenzing \cite[Prop.~2.1]{HLn83} (see also \rmkref{filtered-colimit}). That is, we must show that every homomorphism of $R$-modules $\varphi \colon N \to M$, where $N$ is finitely generated, factors through a maximal CM $R$-module. If we view $N$ and $M$ as modules over $\te$, then $N$ is still finitely generated and $M$ is Gorenstein flat by assumption \eqclbl{iii}. As $\te$ is Gorenstein, \cite[Lem.~10.3.6]{rha} yields that $\varphi$, as a homomorphism of $(\te)$-modules, factors through a finitely generated Gorenstein projective $(\te)$-module $G$. By viewing the hereby obtained factorization $N \to G \to M$ of $\varphi$, in the category of $(\te)$-modules, through the ring homomorphism $R \to \te$, we get a factorization of the original $\varphi$ in the category of $R$-modules. Thus, it remains to argue that $G$ is maximal CM over $R$, i.e.~that $\depth{G} = d$. By Iyengar and Sather-Wagstaff \cite[Lem.~2.8]{SInSSW04} applied to the local ring homomorphism $R \to \te$, we get that $\depth{G} = \depth[\te]{G}$. As $G$ is Gorenstein projective over $\te$, it is also maximal CM over $\te$, see  e.g.~\cite[Cor.~11.5.4]{rha}, so $\depth[\te]{G} = \dim[(\te)]=d$.
\end{proof}

In view of \dfnref{bigCM} (due to Hochster), we suggest the following:

\begin{dfn}
  \label{dfn:wbbCM}
  Let $A$ be a commutative noetherian local ring. An $A$-module $M$ is said to be \emph{weak balanced big CM} if every system of parameters of $A$ is a weak $M$-regular sequence. The category of such $A$-modules is denoted by $\wbbCM$ (where the ring $A$ is understood).
\end{dfn}

\begin{rmk}
  \label{rmk:bbCM-vs-wbbCM}
  Let $(A,\n,\ell)$ be any commutative noetherian local ring. If an $A$-module $M$ satisfies $\n M \neq M$, then $M$ is balanced big CM if (and only if) it is weak balanced big CM. Indeed, under the assumption $\n M \neq M$, a sequence of elements in $\n$ is $M$-regular if (and only if) it is weak $M$-regular.
\end{rmk}

With this terminology, the equivalence of conditions \eqclbl{i} and \eqclbl{ii} in Theorem~B can be expressed as follows: \emph{Over a CM ring $R$ with a dualizing module, a module is weak balanced big CM if and only if it is a direct limit of finitely generated maximal CM modules.} In symbols, the result can be written as
\begin{equation}
  \label{eq:lim}
  \wbbCM \,=\, \varinjlim(\MCM)\;.
\end{equation}

\begin{qst}
  Over a general commutative noetherian local ring (not assumed to be CM with a dualizing module), how are the classes $\wbbCM$ and $\MCM$ related?
\end{qst}

\begin{exa}
  \label{exa:wbbCM}
  As always in this paper, $R$ denotes the ring from \stpref{setup}.
  \begin{prt}
  \item It follows from \prpref{sop} that every flat $R$-module is weak balanced big CM.
  \item Recall that an $R$-module $M$ is said to be \emph{torsion-free} if every non-zerodivisor on $R$ is also a non-zero divisor on $M$. If $R$ has dimension $d=1$, then a system of parameters of $R$ is nothing but a non-zerodivisor on $R$, so in this case ``weak balanced big CM'' just means ``torsion-free''.
  \end{prt}
\end{exa}

Next we prove Theorem~A from the Introduction. 

\begin{proof}[Proof of Theorem~A]
  Let $M$ be a balanced big CM $R$-module. As $M$ is, in particular, weak balanced big CM, it can by Theorem~B be written as $M = \varinjlim M_i$ for some direct system \mbox{$\varphi_{\!ji} \colon M_i \to M_{\!j}$} ($i \leqslant j$) of finitely generated maximal CM $R$-modules. \emph{A priori}, some of the $M_i$'s could be zero, and hence they are not necessarily small CM modules. However, we can define a new direct system by setting $M'_i = M_i \oplus R$ and $\varphi'_{\!ji} = \varphi_{\!ji} \oplus 0$. This direct system clearly has the same direct limit as the original direct system (that is, $M$). And since $R$ is CM, every $M'_i$ is a non-zero maximal CM (that is, a small CM) module.
\end{proof}

The following example is due to Griffith \cite[Rem.~3.3]{Griffith}.

\begin{exa}
  \label{exa:Griffith1}
  Let \mbox{$R=k[\mspace{-2.5mu}[x,y]\mspace{-2.5mu}]$} be the ring of formal power series in two variables $x,y$ with coefficients in a field $k$. It is a regular, and hence a CM, local ring of dimension $d=2$. Set $E=E_R(R/(y))$ and $M=R \oplus E$. Multiplication by $x$ is an automorphism on $E$ since $x \notin (y)$, see e.g.~\cite[Thm.~3.3.8(1)]{rha}, so $x$ is a non-zerodivisor on $M$ with $M/xM \cong R/(x)$. It follows that $y$ is a non-zerodivisor on $M/xM$ and that $M/(x,y)M \cong R/(x,y) \neq 0$. Hence the system of parameters $x,y$ of $R$ is an $M$-regular sequence, so $M$ is a big CM $R$-module.

However, $M$ is not a balanced big CM module since the sequence $y,x$ is not $M$-regular. Indeed, multiplication by $y$ is not a monomorphism on $E$ (the entire submodule $R/(y)$ of $E$ is mapped to zero), and hence $y$ is a zerodivisor on $M$.

As $M$ is not balanced big CM, and since $(x,y)M \neq M$, it follows from   \rmkref{bbCM-vs-wbbCM} that $M$ is not even weak balanced big CM. Theorem~B now shows that 
$M$ can not be written as a direct limit of finitely generated maximal CM $R$-modules.
\end{exa}

\begin{obs}
  \label{obs:Griffith2}
  Recall from \secref{depth} that a finitely generated $R$-module $X$ is maximal CM if and only if $\depth{X} \geqslant d$ (and equality holds if $X \neq 0$). As the functors $\Ext{i}{k}{-}$ commute with direct limits, it follows from the definition of depth and from Theorem~B that for every $R$-module $M$, the following implication holds:
\begin{displaymath}
  \text{$M$ is weak balanced big CM} \quad \Longrightarrow \quad \depth{M} \geqslant d\;.
\end{displaymath}
The converse is not true, as the $R$-module $M$ from \exaref{Griffith1} is not weak balanced big CM, but it does have $\depth{M}=2=d$. Here is one way to see why:

Since $E$ is injective one has $\Ext{i}{k}{E}=0$ for all $i>0$. We also have $\Hom{k}{E}=0$. Indeed, since $k=R/(x,y)$ there is an isomorphism $\Hom{k}{E} \cong \{e \in E \,|\, (x,y)e=0\}$. And if $e \in E$ satisfies $(x,y)e=0$ then, in particular, $xe=0$ which implies that $e=0$ since $x$ is a non-zerodivisor on $E$.
Thus $\depth{E}=\infty$ and $\depth{M}=\depth{(R \oplus E)}=\depthR=2$.

A more explicit way to formulate these considerations are as follows. The $R$-module $M$ from \exaref{Griffith1} has the property that $\Ext{i}{k}{M}=0$ for $i=0,1$; but $M$ is not a direct limit of finitely generated $R$-modules with this property (of course, $M$ is the direct limit of \emph{some} direct system of finitely generated $R$-modules).
\end{obs}

We mention another easy consequence of Theorem~B.

\begin{prp}
  \label{prp:classic}
  The following conditions are equivalent:
  \begin{eqc}
  \item $R$ is regular.

  \item Every weak balanced big CM $\,R$-module is flat.\footnote{\ Recall from  \exaref{wbbCM}(a) that a flat $R$-module is always weak balanced big CM.}
  \end{eqc}
\end{prp}

\begin{proof}
By Lazard \cite[Thm.~1.2]{DLz69} an $R$-module is flat if and only if it is a direct limit of finitely generated projective $R$-modules. And by Cor.~1.4 in \emph{loc.~cit.}~every finitely generated flat $R$-module is projective. By Theorem~B an $R$-module is weak balanced big CM if and only if it is a direct limit of finitely generated maximal CM modules. And every finitely generated weak balanced big CM $R$-module is maximal CM.  It follows that condition \eqclbl{ii} holds if and only if every finitely generated maximal CM module is projective, and this is tantamount to $R$ being regular.
\end{proof}

In view of \exaref{wbbCM}(b) we have the following special case of the implication \mbox{\eqclbl{i}$\,\Rightarrow\,$\eqclbl{ii}} in the result above: \emph{Over a principal ideal domain every torsion-free $R$-module is flat.} This classic result can of course be found in most textbooks on homological algebra; see for example Rotman \cite[Cor.~3.50]{Rotman}.

\section{Applications in relative homological algebra}
\label{sec:rha}

In this final section, we give applications of Theorem~B from the Introduction in relative homological algebra. Special attention is paid to preenvelopes and covers by maximal CM modules and by (weak) balanced big CM modules. First we recall some relevant notions.

  Let $\mathcal{A}$ be a class of objects in a category $\mathcal{M}$, and let $M$ be an object in $\mathcal{M}$. Following Enochs and Jenda \cite[Def.~5.1.1]{rha}, a morphism
$\pi \colon A \to M$ with $A \in \mathcal{A}$ is an \emph{$\mathcal{A}$-precover} of $M$ if every other morphism $\pi' \colon A' \to M$ with $A' \in \mathcal{A}$ factors through $\pi$, as illustrated below.
\begin{displaymath}
  \xymatrix@R=0.8pc{
    A' \ar@{-->}[dd] \ar[dr]^-{\pi'} & {} \\
    {} & M \\
    A \ar[ur]_-{\pi} & {}
   }
\end{displaymath}
An $\mathcal{A}$-precover $\pi \colon A \to M$ is called an \emph{$\mathcal{A}$-cover} if every endomorphism $\varphi$ of $A$ that satisfies $\pi\varphi = \pi$ is an automorphism. The notion of \emph{$\mathcal{A}$-(pre)envelopes} is categorically dual to the notion of $\mathcal{A}$-(pre)covers.\footnote{\ Let $\mathcal{M}=\mathsf{Mod}(A)$ be the category of (left) modules over a ring $A$. If $\mathcal{A}=\mathsf{Prj}(A)$ is the class of projective $A$-modules, then an $\mathcal{A}$-cover is exactly the same as a \emph{projective cover} in the sense of Bass \cite{HBs60}. If $\mathcal{A}=\mathsf{Inj}(A)$ is the class of injective $A$-modules, then an $\mathcal{A}$-envelope is exactly the same as a \emph{injective hull} in the sense of Eckmann and Schopf \cite{EckmannSchopf}. Proofs of these facts can be found in Xu \cite[Thms.~1.2.11 and 1.2.12]{xu}.}

Let $\mathcal{M}$ be an abelian category (in our case, $\mathcal{M}$ will be the category of all modules over some ring or the category of finitely generated modules over a noetherian ring). For a class of objects $\mathcal{A}$ in $\mathcal{M}$ we define
\begin{align*}
  {}^\perp\mathcal{A} &= \{ M \in \mathcal{M} \ |\, \Ext[\mathcal{M}]{1}{M}{A} =0 \, \text{ for all } \, A \in \mathcal{A}\,\}\;, \quad \text{and} 
  \\
  \mathcal{A}^\perp &= \{ M \in \mathcal{M} \ |\, \Ext[\mathcal{M}]{1}{A}{M} =0 \, \text{ for all } \, A \in \mathcal{A}\,\}\;.
\end{align*}
A pair $(\mathcal{A},\mathcal{B})$ of classes of objects in $\mathcal{M}$ is called a \emph{cotorsion pair} if $\mathcal{A}^\perp = \mathcal{B}$ and $\mathcal{A} = {}^\perp\mathcal{B}$. A cotorsion pair $(\mathcal{A},\mathcal{B})$ is called:
\begin{prt}

\item[--] \emph{Hereditary} if  the class $\mathcal{A}$ is \emph{resolving}; this means that $\mathcal{A}$ contains all projective objects in $\mathcal{M}$ and that $\mathcal{A}$ is closed under extensions and kernels of epimorphisms. See 
\cite[Def.~2.2.8(i) and Lem.~2.2.10]{GobelTrlifaj} for further details.

\item[--] \emph{Complete} if the class $\mathcal{A}$ has \emph{enough projectives}; this means that for every $M \in \mathcal{M}$ there exists an exact sequence $0 \to B \to A \to M \to 0$ with $A \in \mathcal{A}$ and $B \in \mathcal{B}$. Equivalently, $\mathcal{B}$ has \emph{enough injectives}, that is for every $M \in \mathcal{M}$ there is an exact sequence $0 \to M \to B \to A \to 0$ with $A \in \mathcal{A}$ and $B \in \mathcal{B}$. See \cite[\S7.1]{rha} and \cite[Lem.~2.2.6]{GobelTrlifaj}. 

Note that if $(\mathcal{A},\mathcal{B})$ is a complete cotorsion pair, then every $M \in \mathcal{M}$ has a (special) $\mathcal{A}$-precover and a (special) $\mathcal{B}$-preenvelope.

\item[--] \emph{Perfect} if every  $M \in \mathcal{M}$ has an $\mathcal{A}$-cover and a $\mathcal{B}$-envelope. See \cite[Def.~2.3.1]{GobelTrlifaj}.

\end{prt}

Recall that $\MCM$ denotes the category of finitely generated maximal CM modules; see \secref{depth}. We denote by $\I$ the categeory of finitely generated modules with finite injective dimension. As always, $R$ is the ring from \stpref{setup}.

A classic result of Ischebeck \cite{FIs69} (see also \cite[Exerc. 3.1.24]{BruHer}) shows that $\Ext{i}{M}{I}=0$ for all $M \in \MCM$, all $I \in \I$, and all $i>0$. In particular, one has $\MCM^\perp \supseteq \I$ and $\MCM \subseteq {}^\perp\I$. Combining this fact with the existence of \emph{maximal CM approximations} and \emph{hulls of finite injective dimension}, proved by Auslander and Buchweitz in \cite[Thm.~A]{MAsROB89}, one gets:

\begin{thm}
  \label{thm:MCM-cotorsion-pair}
On the abelian category of finitely generated $R$-modules, the pair $(\MCM,\I)$ is a complete hereditary cotorsion pair. In particular, every finitely generated $R$-module has a $\MCM$-precover and an $\I$-preenvelope. \qed
\end{thm}

We now prove Theorems~C and D from the Introduction.

\begin{proof}[Proof of Theorem~C]
By Crawley-Boevey \cite[Thm.~(4.2)]{WCB94}, the assertion is equivalent to \smash{$\varinjlim(\MCM)$} being closed under products. By Theorem~B, the class \smash{$\varinjlim(\MCM)$} is exactly the class of weak balanced big CM $R$-modules, and it follows from \prpref{sop} and \cite[Thm.~3.2.26]{rha} that this class is closed under products.
\end{proof}

\begin{proof}[Proof of Theorem~D]
It follows from \prpref{sop} that $\wbbCM$ is resolving, and that it is closed under products, coproducts, and direct summands. Since \smash{$\wbbCM = \varinjlim(\MCM)$}, see \eqref{lim}, Lenzing \cite[Prop.~2.2]{HLn83} shows that $\wbbCM$ is closed under pure submodules and pure quotient modules. Thus \cite[Thm.~3.4]{HJ08} yields that $(\wbbCM,\wbbCM^\perp)$ is a perfect cotorsion pair, which is hereditary as $\wbbCM$ is resolving.
It follows from Rada and Saorin \cite[Cor.~3.5(c)]{RadaSaorin} that every $R$-module has a $\wbbCM$-preenvelope. 
\end{proof}

We end this paper by proving the existence of (non-weak) balanced big CM covers for certain types of modules.

\begin{prp}
  \label{prp:bbCM-cover}
  Let $M$ be any $R$-module. If $\m M \neq M$, then $M$ has a surjective cover with respect to the class of balanced big CM $R$-modules.
\end{prp}

\begin{proof}
  It follows from Theorem~D that $M$ has a cover with respect to the class of \emph{weak} balanced big CM $R$-modules, say, \mbox{$\pi \colon W \to M$} where $W$ is in $\wbbCM$. The homomorphism $\pi$ must be surjective since every projective $R$-module belongs to $\wbbCM$. As $\pi$ is surjective, so is the induced map $W/\m W \to M/\m M$, and it follows that $\m W \neq W$. Hence $W$ is balanced big CM $R$-modules, see \rmkref{bbCM-vs-wbbCM}, and $\pi$ is the desired cover.
\end{proof}

\prpref{bbCM-cover} is related to the main result in \cite[Thm.~5.6]{Simon09} by Simon. This result asserts that, over a CM ring with a dualizing module, every \emph{complete} module has a surjective cover w.r.t.~the class $\mathcal{X}$ of \emph{complete} big CM modules \emph{including} the zero module.

In \prpref{bbCM-cover} the assumption $\m M \neq M$ is essential; here is a pathological example.

\begin{exa}
  \label{exa:bbCM-cover}
  The zero module $M=0$ does not have a cover with respect to the class of balanced big CM $R$-modules. Indeed, suppose that \mbox{$W \to 0$} is such a cover. Since the zero endo\-morphism \mbox{$0 \colon W \to W$} makes the diagram
\begin{displaymath}
  \xymatrix@R=0.75pc{
    W \ar[dd]_-{0} \ar[dr] & {} \\
   {} & 0 \\
    W \ar[ur] & {}
  }
\end{displaymath}
commutative, it follows from the definition of a cover that \mbox{$0 \colon W \to W$} is an automorphism. This means that $W=0$, which is impossible as $W$ is a big CM module.
\end{exa}

\def\cprime{$'$}
  \providecommand{\arxiv}[2][AC]{\mbox{\href{http://arxiv.org/abs/#2}{\sf
  arXiv:#2 [math.#1]}}}
  \providecommand{\oldarxiv}[2][AC]{\mbox{\href{http://arxiv.org/abs/math/#2}{\sf
  arXiv:math/#2
  [math.#1]}}}\providecommand{\MR}[1]{\mbox{\href{http://www.ams.org/mathscinet-getitem?mr=#1}{#1}}}
  \renewcommand{\MR}[1]{\mbox{\href{http://www.ams.org/mathscinet-getitem?mr=#1}{#1}}}
\providecommand{\bysame}{\leavevmode\hbox to3em{\hrulefill}\thinspace}
\providecommand{\MR}{\relax\ifhmode\unskip\space\fi MR }
\providecommand{\MRhref}[2]{%
  \href{http://www.ams.org/mathscinet-getitem?mr=#1}{#2}
}
\providecommand{\href}[2]{#2}

\end{document}